\documentclass[11pt]{amsart}%
\usepackage{graphicx}
\usepackage{amscd}
\usepackage{amsmath}
\usepackage{amsfonts}
\usepackage{amssymb}%
\providecommand{\U}[1]{\protect\rule{.1in}{.1in}}
\newtheorem{theorem}{Theorem}
\theoremstyle{plain}

\newtheorem{lemma}{Lemma}[section]

\newtheorem{remark}{Remark}[section]

\numberwithin{equation}{section}
\numberwithin{theorem}{section}

\begin{document}
\title[Spectral Analysis of the Sixth-order Krall Differential Expression]{A New Spectral Analysis of the Sixth-order Krall Differential Expression}
\author{K. Elliott}
\address{Department of Mathematics, University of Wisconsin-Eau Claire, Hibbard
Humanities Hall 508, Eau Claire, WI 54701. }
\email{elliotkl@uwec.edu}
\author{L. L. Littlejohn}
\address{Department of Mathematics, Baylor University, One Bear Place \#97328, Waco, TX 76798-7328.}
\email{lance\_littlejohn@baylor.edu}
\author{R. Wellman}
\address{Department of Mathematics and Computer Science, Tut Science 208, 14 E. Cache
La Poudre St., Colorado College, Colorado Springs, CO 80903.}
\email{rwellman@coloradocollege.edu}
\date{January 22, 2020; [Elliott-Littlejohn-WellmanAIOT.tex]; submitted to AIOT on 1/22/2010}
\dedicatory{We are pleased to dedicate this paper to Dr. Franciszek Hugon Szafraniec on
occasion of his 80th birthday for his significant contributions to
mathematics. Franek, a scholar of eminence, is a long-time colleague and
friend of one of the authors (LLL). }

\begin{abstract}
In this paper, we construct a self-adjoint operator $\widehat{T}_{A,B}$
generated by the sixth-order Krall differential expression in the extended
Hilbert space $L^{2}(-1,1)\oplus\mathbb{C}^{2}.$ To obtain $\widehat{T}%
_{A,B},$ we apply a new general theory, the so-called GKN-EM theory, developed
recently by Littlejohn and Wellman that extends the classical
Glazman-Krein-Naimark theory using a complex symplectic geometric approach
developed by Everitt and Markus. This work extends earlier studies of the
Krall expression by both\ Littlejohn and Loveland.

\end{abstract}
\maketitle

\section{Introduction}

In 1981, Littlejohn (see \cite{Littlejohn-thesis} and \cite{Littlejohn-QM})
discovered the sixth-order Krall differential expression%
\begin{align}
\ell_{K}[y](x):=  &  (x^{2}-1)^{3}y^{(6)}+18x(x^{2}-1)^{2}y^{(5)}\nonumber\\
&  +(x^{2}-1)\left(  (3A+3B+96)x^{2}-3A-3B-36\right)  y^{(4)}%
+(24A+24B+168)x(x^{2}-1)y^{\prime\prime\prime}\nonumber\\
&  +\left(  (12AB+42A+42B+72)x^{2}+(12B-12A)x-12AB-30A-30B-72\right)
y^{\prime\prime}\label{I-1}\\
&  +\left(  (24AB+12A+12B)x+12B-12A\right)  y^{\prime};\nonumber
\end{align}
here $A,B$ are fixed, positive constants and $x\in(-1,1)$. The expression
(\ref{I-1}) is Lagrangian symmetric; in fact,
\begin{align}
\ell_{K}[y](x)=  &  -\left(  (1-x^{2})^{3}y^{\prime\prime\prime}\right)
^{\prime\prime\prime}+\left(  (1-x^{2})\left(  12+(3A+3B+6)(1-x^{2})\right)
y^{\prime\prime}\right)  ^{\prime\prime}\label{I-2}\\
&  -\left(  \left(  (6A-6B-12AB)x^{2}+(12A-12B)x+12AB+18A+18B+24\right)
y^{\prime}\right)  ^{\prime}.\nonumber
\end{align}
If $n\in\mathbb{N}_{0}=\{0,1,2,\ldots\}$ and
\begin{equation}
\lambda_{n}:=n(n-1)(n^{4}+2n^{3}+(3A+3B-1)n^{2}+(3A+3B-2)n+12AB), \label{I-3}%
\end{equation}
the eigenvalue equation%
\begin{equation}
\ell_{K}[y](x)=\lambda_{n}y(x)\quad(n\in\mathbb{N}_{0}), \label{I-4}%
\end{equation}
has a polynomial solution of degree exactly $n,$ given explicitly by
\begin{equation}
y(x)=K_{n,A,B}(x):=\sum_{j=0}^{n}\frac{(-1)^{[\frac{j}{2}]}%
(2n-j)!Q(n,j)x^{n-j}}{2^{n+1}\left(  n-\left[  (j+1)/2\right]  \right)
!\left[  j/2\right]  !(n-j)!(n^{2}+n+A+B)}, \label{I-5}%
\end{equation}
where $[\cdot]$ denotes the greatest integer function and
\begin{align*}
Q(n,j):=  &  \frac{2+(-1)^{j}}{2}\left(  (n^{4}+(2A+2B-1)n^{2}+4AB+2j(n^{2}%
+n+A+B\right) \\
&  +\frac{1-(-1)^{j}}{2}(4B-4A).
\end{align*}
The set $\{K_{n,A,B}\}_{n=0}^{\infty}$ is called the Krall polynomials, named
after H. L. Krall (1907-1994).

For perspective, Krall \cite{Krall 1938} discovered a fourth-order
differential expression in 1938 which was subsequently named the Legendre type
expression and studied by his son A. M. Krall (1936-2008) in \cite{Krall-1981}%
. The Legendre type expression is given explicitly by
\begin{align}
\ell_{L}[y](x)  &  :=\left(  (1-x^{2})^{2}y^{\prime\prime}\right)
^{\prime\prime}-\left(  \left(  8+4A(1-x^{2})\right)  y^{\prime}\right)
^{\prime}\label{I-5.5}\\
&  =(1-x^{2})^{2}y^{(4)}+8x(x^{2}-1)y^{\prime\prime\prime}+(4A+12)(x^{2}%
-1)y^{\prime\prime}+8Axy^{\prime};\nonumber
\end{align}
in this case, $A$ is a fixed, positive parameter and $x\in(-1,1).$ The
eigenvalue equation%
\[
\ell_{L}[y](x)=\mu_{n}y(x)\quad(n\in\mathbb{N}_{0}),
\]
where
\[
\mu_{n}=n(n+1)(n^{2}+n+4A-2),
\]
has a polynomial solution%
\[
y(x)=P_{n,A}(x):=\sum_{j=0}^{[n/2]}\frac{(-1)^{j}(2n-2j)!\left(
A+n(n-1)/2+2j\right)  }{2^{n}j!(n-j)!(n-2j)!}x^{n-2j}%
\]
called the $n^{th}$ Legendre type polynomial. In this case, the sequence
$\{P_{n,A}\}_{n=0}^{\infty}$ forms a complete orthogonal set in the Hilbert
space $L_{\mu}^{2}[-1,1]$ with inner product%
\begin{equation}
(f,g)_{A}:=\frac{f(-1)\overline{g}(-1)}{A}+\int_{-1}^{1}f(x)\overline
{g}(x)dx+\frac{f(1)\overline{g}(1)}{A}. \label{I-6}%
\end{equation}
Here $d\mu$ is the positive Borel measure generated from the monotonically
non-decreasing function%
\[
\widehat{\mu}(x):=\left\{
\begin{array}
[c]{ll}%
-1-\frac{1}{A} & \text{if }x\leq-1\\
x & \text{if }-1<x<1\\
1+\frac{1}{A} & \text{if }x\geq1.
\end{array}
\right.
\]
Notice that, in (\ref{I-6}), there are equal `jumps' at the endpoints
$x=\pm1.$ In 1979, H. L. Krall asked the questions: if there are
\textit{unequal} jumps at $x=\pm1,$ what are the underlying orthogonal
polynomials and is there an accompanying differential equation that has these
polynomials as eigenfunctions? The answers are the Krall polynomials
(\ref{I-5}) and the sixth-order Krall equation (\ref{I-4}). The Krall
polynomials $\{K_{n,A,B}\}_{n=0}^{\infty}$ form a complete orthogonal set in
the Hilbert space $L_{\kappa}^{2}[-1,1]$ with inner product%
\begin{equation}
(f,g)_{\kappa}=\int_{[-1,1]}f\overline{g}d\kappa:=\frac{f(-1)\overline{g}%
(-1)}{A}+\int_{-1}^{1}f(x)\overline{g}(x)dx+\frac{f(1)\overline{g}(1)}{B}.
\label{I-7}%
\end{equation}
where $\kappa$ is the positive Borel measure generated from the distribution
function%
\[
\widehat{\kappa}(x)=\left\{
\begin{array}
[c]{ll}%
-1-\frac{1}{A} & \text{if }x\leq-1\\
x & \text{if }-1<x<1\\
1+\frac{1}{B} & \text{if }x\geq1.
\end{array}
\right.
\]

The first systematic spectral analysis of (\ref{I-1}) was accomplished by
Susan Loveland in her Ph.D. thesis \cite{Loveland} (see also
\cite{Everitt-Littlejohn-Loveland}) where Loveland constructed the
self-adjoint operator $T_{A,B}$ in $L_{\kappa}^{2}[-1,1],$ generated by
$\ell_{K}[\cdot],$ having the Krall polynomials as eigenfunctions. The
analysis was challenging; indeed, the classical Glazman-Krein-Naimark (GKN)
theory (see \cite[Chapter V]{Naimark}) is not immediately applicable so new
techniques and analytic tools had to be developed. It is noteworthy that the
Krall expression (\ref{I-1}) is in the limit-5 case at both endpoints $x=\pm1$
in $L^{2}(-1,1).$ Consequently, \underline{four} appropriate boundary
conditions (separated and/or non-separated) are needed to describe each
self-adjoint operator, generated by $\ell_{K}[\cdot]$, in $L^{2}(-1,1).$ In
this $L^{2}(-1,1)$ setting, the GKN theory can be used. However, we wish to
study $\ell_{K}[\cdot]$ in $L_{\kappa}^{2}[-1,1]$ (since this is where the
Krall polynomial eigenfunctions are orthogonal) and, in this setting, the GKN
theory fails to explicitly help. As it turns out, it is remarkable that only
\underline{two} boundary conditions are needed to describe the self-adjoint
operator $T_{A,B},$ generated by $\ell_{K}[\cdot],$ in $L_{\kappa}^{2}[-1,1]$
having the Krall polynomials $\{K_{n,A,B}\}_{n=0}^{\infty}$ as eigenfunctions.

The initial work of Everitt, Littlejohn and Loveland from 1990-1993 on the
spectral analysis of the Krall expression did not address the question of
finding \textit{all }self-adjoint operators $T$ in $L_{\kappa}^{2}[-1,1]$
generated by $\ell_{K}[\cdot].$ Quite simply, at that time, the authors did
not have the tools to answer this more general question. Now, however, due to
a recent generalization of the GKN theory recently developed by Littlejohn and
Wellman \cite{Littlejohn-Wellman2019}, we have the tools to find all
self-adjoint operators $\widehat{T}$ in $L_{\kappa}^{2}[-1,1]$ generated by
$\ell_{K}[\cdot].$ The key to this generalization is a series of important
papers by W. N. Everitt and L. Markus \cite{Everitt-Markus1999},
\cite{Everitt-MarkusBVP}, \cite{Everitt-MarkusComplexSymplectic} who saw
important connections between the von Neumann theory of self-adjoint
extensions of symmetric operators and the theory of complex symplectic geometry.

The contents of this paper are as follows. In Section \ref{Section Two}, we
review properties of functions in the maximal domain $\Delta_{K}$ of the Krall
expression $\ell_{K}[\cdot]$ in $L^{2}(-1,1)$. In Section \ref{Section Three},
we discuss a dense subset $\delta_{K}$ $\subset\Delta_{K}$ that is the domain
of the self-adjoint operator $T_{A,B}$ in $L_{\kappa}^{2}[-1,1]$ generated by
$\ell_{K}[\cdot],$ having the Krall polynomials $\{K_{n,A,B}\}_{n=0}^{\infty}$
as eigenfunctions. We also briefly describe the Everitt-Littlejohn-Loveland
method for constructing this self-adjoint operator. Section \ref{Section Four}
gives a synopsis of the GKN-EM theory developed by Littlejohn and Wellman
based on the work of Everitt and Markus. In this regard, it is important to
note that $L_{\kappa}^{2}[-1,1]$ is isometrically isomorphic to $L^{2}%
(-1,1)\oplus\mathbb{C}^{2}.$ Lastly, in Section \ref{Section Five}, we apply
the GKN-EM theory to construct $\widehat{T}_{A,B}$ which is equivalent to the
operator $T_{A,B}$ found by Loveland. As the reader will see, this approach is
algorithmic in the sense that the GKN-EM Theorem (see Theorem \ref{GKNEM}) may
be applied to find all self-adjoint operators $\widehat{T}$ in $L_{\kappa}%
^{2}[-1,1]$ generated by $\ell_{K}[\cdot].$

\section{Properties of Functions in the Maximal Domain $\Delta_{K}$ of
$\ell_{K}[\cdot]$ in $L^{2}(-1,1)$\label{Section Two}}

The maximal domain $\Delta_{\kappa}$ (see \cite[Chapter V]{Naimark}) for
$\ell_{K}[\cdot]$ in $H:=L^{2}(-1,1)$ is defined to be%
\begin{align}
\Delta_{K}:=\{f:(-1,1)\rightarrow\mathbb{C\mid}  &  f^{(j)}\in
AC_{\mathrm{loc}}(-1,1)\text{ }(j=0,1,2,3,4,5);\label{II-1}\\
&  f,\ell_{K}[f]\in L^{2}(-1,1)\}.\nonumber
\end{align}
This set $\Delta_{K}$ is the largest subspace of functions $f$ in
$L^{2}(-1,1)$ for which $\ell_{K}[f]$ $\in L^{2}(-1,1).$ In this setting,
Green's formula reads as%
\begin{align}
\lbrack f,g]_{H}  &  :=\int_{-1}^{1}\ell_{K}[f](x)\overline{g}(x)dx-\int%
_{-1}^{1}f(x)\ell_{K}[\overline{g}](x)dx\label{II-1.5}\\
&  =[f,g]_{K}(x)_{-1}^{1}:=[f,g]_{K}(1)-[f,g]_{K}(-1)\quad(f,g\in\Delta
_{K}),\nonumber
\end{align}
where $[\cdot,\cdot]_{K}$ is the symplectic (bilinear) form or bilinear
concomitant defined by%
\begin{align}
\lbrack f,g]_{K}(x)  &  :=\left(  -\left(  (1-x^{2})^{3}f^{(3)}(x)\right)
^{\prime\prime}+\left(  (1-x^{2})(12+\alpha(1-x^{2}))f^{\prime\prime
}(x)\right)  ^{\prime}-\pi(x)f^{\prime}(x)\right)  \overline{g}(x)\nonumber\\
&  -\left(  -\left(  (1-x^{2})^{3}\overline{g}^{(3)}(x)\right)  ^{\prime
\prime}+\left(  (1-x^{2})(12+\alpha(1-x^{2}))\overline{g}^{\prime\prime
}(x)\right)  ^{\prime}-\pi(x)\overline{g}^{\prime}(x)\right)  f(x)\nonumber\\
&  -\left(  -\left(  (1-x^{2})^{3}f^{(3)}(x)\right)  ^{\prime}+(1-x^{2}%
)(12+\alpha(1-x^{2}))f^{\prime\prime}(x)\right)  \overline{g}^{\prime
}(x)\label{II-2}\\
&  +\left(  -\left(  (1-x^{2})^{3}\overline{g}^{(3)}(x)\right)  ^{\prime
}+(1-x^{2})(12+\alpha(1-x^{2}))\overline{g}^{\prime\prime}(x)\right)
f^{\prime}(x)\nonumber\\
&  -(1-x^{2})^{3}(f^{(3)}(x)\overline{g}^{\prime\prime}(x)-f^{\prime\prime
}(x)\overline{g}^{(3)}(x)).\nonumber
\end{align}
Here, for simplicity, $\alpha:=3A+3B+6$ and
\[
\pi(x):=(-6A-6B-12AB)x^{2}+(12A-12B)x+(12AB+18A+18B+24).
\]
We emphasize to the reader the subtle difference between the symplectic form
$[\cdot,\cdot]_{H}$ and the bilinear concomitant $[\cdot,\cdot]_{K}.$

The associated maximal operator $T:\mathcal{D}(T)\subset L^{2}%
(-1,1)\rightarrow L^{2}(-1,1)$ is given by
\begin{align*}
T[f](x)  &  =\ell_{K}[f](x)\\
f\in\mathcal{D}(T):  &  =\Delta_{K}.
\end{align*}
It is a densely defined operator with adjoint $T^{\ast}=T_{0}:\mathcal{D}%
(T_{0})\subset L^{2}(-1,1)\rightarrow L^{2}(-1,1),$ called the minimal
operator. From the definition of $\Delta_{K},$ and H\"{o}lder's inequality, we
see that
\begin{equation}
\lim_{x\rightarrow\pm1}[f,g]_{K}(x) \label{Limit of [f,g] at 1 and -1}%
\end{equation}
exists and is finite for all $f,g\in\Delta_{K}.$ Moreover, $1\in\Delta_{K}$
and
\begin{equation}
\lbrack f,1]_{K}(x)=-\left(  (1-x^{2})^{3}f^{(3)}(x)\right)  ^{\prime\prime
}+\left(  (1-x^{2})(12+\alpha(1-x^{2}))f^{\prime\prime}(x)\right)  ^{\prime
}-\pi(x)f^{\prime}(x) \label{II-3}%
\end{equation}
for all $f\in\Delta_{K}.$

For $f\in\Delta_{K}$ and $x\in(-1,1),$ define $\Lambda:\Delta_{K}\rightarrow
L^{2}(-1,1)$ by%
\begin{align}
\Lambda\lbrack f](x)  &  :=\int_{0}^{x}\int_{0}^{t}\ell_{K}[f](s)dsdt-f^{(5)}%
(0)x-f^{(4)}(0)+(18+\alpha)f^{\prime\prime\prime}(0)x\label{II-4}\\
&  +(12+\alpha)f^{\prime\prime}(0)-\pi(0)f^{\prime}(0)x+\int_{0}^{x}%
\pi(t)f^{\prime}(t)dt\nonumber\\
&  =-\left(  (1-x^{2})^{3}f^{\prime\prime\prime}(x)\right)  ^{\prime}%
+(1-x^{2})(12+\alpha(1-x^{2}))f^{\prime\prime}(x).\nonumber
\end{align}
Note that $\Lambda\lbrack f]\in L^{2}(-1,1)$ and $\Lambda\lbrack f]\in
AC_{\mathrm{loc}}(-1,1)$ for all $f\in\Delta_{K}$. As a result of Theorem
\ref{Properties of Maximal Domain} below, we will see that
\[
\lim_{x\rightarrow\pm1}\Lambda\lbrack f](x)
\]
exists and is finite and so $\Lambda\lbrack f]\in AC[-1,1]$ for $f\in
\Delta_{K}$. As we will see in the next section, this linear operator
$\Lambda\lbrack\cdot]$ plays a key role in our analytic study of $\ell
_{K}[\cdot]$ in $L^{2}(-1,1).$

Loveland \cite[Theorem 7.2.2]{Loveland} (see also
\cite{Everitt-Littlejohn-Loveland}) established the following theorem that
lists properties of functions in $\Delta_{K}.$

\begin{theorem}
\label{Properties of Maximal Domain}Let $f,g\in\Delta_{K}.$ Then

\begin{enumerate}
\item[(i)] $f^{\prime}\in L^{2}(-1,1)$ so $f\in AC[-1,1];$ in particular,
$\Delta_{K}\subset L_{\kappa}^{2}[-1,1]$ and $\Lambda\lbrack f]\in
AC[-1,1];\medskip$

\item[(ii)] $\lim_{x\rightarrow\pm1}(1-x^{2})^{j}f^{(j)}(x)=0$ for
$j=1,2,3;\medskip$

\item[(iii)] $1\in\Delta_{K}$ and $\lim_{x\rightarrow\pm1}[f,1]_{K}%
(x)=\lim_{x\rightarrow\pm1}\left(  \Lambda^{\prime}[f](x)-\pi(x)f^{\prime
}(x)\right)  ;\medskip$

\item[(iv)] $1-x^{2}\in\Delta_{K}$ and
\begin{align*}
\lim_{x\rightarrow1}[f,1-x^{2}]_{K}(x)  &  =2\Lambda\lbrack
f](1)-48(A+2)f(1),\\
\lim_{x\rightarrow-1}[f,1-x^{2}]_{K}(x)  &  =-2\Lambda\lbrack
f](-1)+48(B+2)f(-1).
\end{align*}

\item[(v)] $(1-x^{2})^{2}\in\Delta_{K}$ and
\[
\lim_{x\rightarrow\pm1}[f,(1-x^{2})^{2}]_{K}(x)=\pm192f(\pm1);\medskip
\]

\item[(vi)] Define $h_{\pm}\in C^{6}[-1,1]$ by
\begin{align*}
h_{+}(x)  &  =%
\begin{cases}
0 & x\text{ near }-1\\
\frac{1}{8}(A+2)(1-x^{2})^{2}\ln(1-x^{2})+\frac{1}{2}(1-x^{2})\ln(1-x^{2}) &
x\text{ near }1,
\end{cases}
\\
h_{-}(x)  &  =%
\begin{cases}
\frac{1}{8}(B+2)(1-x^{2})^{2}\ln(1-x^{2})+\frac{1}{2}(1-x^{2})\ln(1-x^{2}) &
x\text{ near }-1\\
0 & x\text{ near }1.
\end{cases}
\end{align*}
Then $h_{\pm}\in\Delta_{K}$; moreover,
\begin{align*}
\lim_{x\rightarrow-1}[f,h_{-}]_{K}(x)  &  =(-32B-12A+16)f(-1)\\
&  +\lim_{x\rightarrow-1}\left(  -\Lambda\lbrack f](x)h_{-}^{\prime
}(x)+32f^{\prime}(x)-(1-x^{2})^{3}(f^{\prime\prime\prime}(x)h_{-}%
^{\prime\prime}(x)-h_{-}^{\prime\prime\prime}(x)f^{\prime\prime}(x))\right)  ,
\end{align*}
and
\begin{align*}
\lim_{x\rightarrow1}[f,h_{+}]_{K}(x)  &  =(32A+12B-16)f(1)\\
&  +\lim_{x\rightarrow1}\left(  -\Lambda\lbrack f](x)h_{+}^{\prime
}(x)+32f^{\prime}(x)-(1-x^{2})^{3}(f^{\prime\prime\prime}(x)h_{+}%
^{\prime\prime}(x)-h_{+}^{\prime\prime\prime}(x)f^{\prime\prime}(x))\right)  ;
\end{align*}

\item[(vii)] $\lim_{x\rightarrow\pm1}[f,(1-x^{2})^{3}]_{K}(x)=0;$ consequently
$(1-x^{2})^{3}$ is in the domain of minimal operator in $L^{2}(-1,1)$
associated with $\ell_{K}[\cdot];\medskip$

\item[(viii)] $\lim_{x\rightarrow\pm1}[f,g]_{K}(x)=[f,1]_{K}(\pm1)\overline
{g}(\pm1)-[\overline{g},1]_{K}(\pm1)f(\pm1)\medskip$\newline$+\lim
_{x\rightarrow\pm1^{\mp}}(-\Lambda\lbrack f](x)\overline{g}^{\prime
}(x)+\Lambda\lbrack\overline{g}](x)f^{\prime}(x)-(1-x^{2})^{3}(f^{(3)}%
(x)\overline{g}^{\prime\prime}(x)-f^{\prime\prime}(x)\overline{g}%
^{\prime\prime\prime}(x)).$ $\blacksquare$
\end{enumerate}
\end{theorem}

\begin{remark}
The importance of the functions $h_{\pm}$ will be made clear in Section
\ref{Section Five} $($see Lemma \ref{GKN Set for T_0 hat}$)$.
\end{remark}

We turn our attention to a brief study of the Frobenius solutions to
\[
\ell_{K}[y](x)=0\quad(x\in(-1,1)).
\]
The endpoints $x=\pm1$ are regular singular endpoints of $\ell_{K}[\cdot]$ in
the sense of Frobenius. At either endpoint, the Frobenius indicial equation is
given by%
\[
\rho(r)=(r-3)(r-2)(r-1)^{2}r(r+1)=0.
\]
A careful analysis of the six linearly independent Frobenius solutions at
$x=1$ yields%
\begin{align*}
\varphi_{3}(x)  &  =\sum_{n=0}^{\infty}a_{n}(x-1)^{n+3}\quad(a_{0}\neq0)\\
\varphi_{2}(x)  &  =\sum_{m=0}^{\infty}b_{n}(x-1)^{n+2}+\log\left\vert
x-1\right\vert \sum_{n=1}^{\infty}\widetilde{b}_{n}(x-1)^{n+2}\quad(b_{0}%
\neq0)\\
\varphi_{1}(x)  &  =\sum_{n=0}^{\infty}c_{n}(x-1)^{n+1}\quad(c_{0}\neq0)\\
\widehat{\varphi}_{1}(x)  &  =3\log\left\vert x-1\right\vert \sum
_{n=0}^{\infty}c_{n}(x-1)^{n+1}+\sum_{n=0}^{\infty}d_{n}(x-1)^{n+1}\quad
(d_{0}\neq0)\\
\varphi_{0}(x)  &  =\sum_{n=0}^{\infty}e_{n}(x-1)^{n}+\log\left\vert
x-1\right\vert \sum_{n=1}^{\infty}f_{n}(x-1)^{n}\quad(e_{0}\neq0)\\
\varphi_{-1}(x)  &  =\log\left\vert x-1\right\vert \sum_{n=1}^{\infty}%
g_{n}(x-1)^{n-1}+\sum_{n=0}^{\infty}h_{n}(x-1)^{n-1}\quad(h_{0}\neq0).
\end{align*}
The subscripts in the above solutions correspond to the indicial root. Each of
these series converges for $\left\vert x-1\right\vert <2.$ Since $\varphi
_{3},\varphi_{2,}\varphi_{1,}\widehat{\varphi}_{1},\varphi_{0}\in L^{2}[0,1)$
but $\varphi_{-1}\notin L^{2}[0,1),$ we see that $\ell_{K}[\cdot]$ is in the
limit-5 case at $x=1$ in $L^{2}(-1,1).$ A similar Frobenius analysis shows
that $\ell_{K}[\cdot]$ is also in the limit-5 case at $x=-1$ in $L^{2}%
(-1,1).$Moreover, it is the case that $\widehat{\varphi}_{1}^{\prime\prime
}\notin L^{2}(-1,1)$ so the smoothness condition given in part (i) of Theorem
\ref{Properties of Maximal Domain} is best possible.

We seek to `extract' from $\Delta_{K}$ a dense (in $L^{2}(-1,1))$ subset that
is the domain of a self-adjoint operator, generated by $\ell_{K}[\cdot],$
having the Krall polynomials as eigenfunctions. To this end, we note that the
Krall polynomials are natural generalizations of both the classical Legendre
polynomials and the Legendre type polynomials. Moreover, the self-adjoint
operators $T_{L}$ and $T_{A}$ generated, respectively, by the classical
second-order Legendre differential expression%
\[
\ell_{L}[y](x)=-((1-x^{2})y^{\prime}(x))^{\prime}\quad(x\in(-1,1))
\]
and the fourth-order Legendre type expression (\ref{I-5.5}), having the
Legendre polynomials and Legendre type polynomials, respectively, as
eigenfunctions\ have the remarkable smoothness properties%
\begin{align*}
f  &  \in\mathcal{D}(T_{L})\Longrightarrow f^{\prime}\in L^{2}(-1,1)\\
f  &  \in\mathcal{D}(T_{A})\Longrightarrow f^{\prime\prime}\in L^{2}(-1,1).
\end{align*}
It is natural to ask: does the self-adjoint operator $T_{A,B}$ in $L_{\kappa
}^{2}[-1,1]$, generated by the Krall expression $\ell_{K}[\cdot]$ and having
the Krall polynomials $\{K_{n,A,B}\}_{n=0}^{\infty}$ as eigenfunctions, have a
similar property? More specifically, will it be the case that%
\[
f\in\mathcal{D}(T_{A,B})\Longrightarrow f^{\prime\prime\prime}\in
L^{2}(-1,1)?
\]
The answer is yes. In the next section, we further restrict the maximal domain
to a proper subset $\delta_{K}\subset\Delta_{K}\subset L_{\kappa}^{2}[-1,1]$
which turns out to be the domain of $T_{A,B}$, generated by $\ell_{K}[\cdot],$
having the Krall polynomials as eigenfunctions.

\section{Properties of Functions in a Proper Subspace $\delta_{K}$ of
$\Delta_{K}$ of $\ell_{K}[\cdot]$ in $L^{2}(-1,1)$\label{Section Three}}

For specific details of results in this section, we refer the reader to
\cite{Everitt-Littlejohn-Loveland} and to the Loveland thesis \cite{Loveland}.

A key idea in these publications to finding $\delta_{K}$ was to restrict the
maximal domain $\Delta_{K}$ so that those Frobenius solutions $\varphi$ of
$\ell_{K}[y]=0$ near $x=\pm1$ satisfying $\varphi^{\prime\prime\prime}\notin
L^{2}(-1,1)$ were eliminated. By taking an appropriate linear combination of
the functions $(1-x^{2})$ and $(1-x^{2})^{2}$ and appealing to Theorem
\ref{Properties of Maximal Domain} (iv) and (v), we were led to defining
$\psi_{\pm}\in C^{6}[-1,1]\cap\Delta_{K}$ by%
\begin{equation}
\psi_{+}(x)=\left\{
\begin{array}
[c]{ll}%
\frac{1}{2}(1-x^{2})+\frac{1}{8}(A+2)(1-x^{2})^{2} & x\text{ near }1\\
0 & x\text{ near }-1
\end{array}
\right.  \label{psi+}%
\end{equation}
and%
\begin{equation}
\psi_{-}(x)=\left\{
\begin{array}
[c]{ll}%
0 & x\text{ near }1\\
-\frac{1}{2}(1-x^{2})-\frac{1}{8}(B+2)(1-x^{2})^{2} & x\text{ near }-1.
\end{array}
\right.  \label{psi-}%
\end{equation}
It is straightforward to see that

\begin{theorem}
\label{Boundary Conditions and Lambda}Let $f\in\Delta_{K}.$ Then%
\begin{equation}
\lbrack f,\psi_{\pm}]_{K}(\pm1)=\Lambda\lbrack f](\pm1).\blacksquare
\label{GKN BC's}%
\end{equation}

\end{theorem}

With this result, we define%
\begin{equation}
\delta_{K}:=\{f\in\Delta_{K}\mid\lbrack f,\psi_{+}](1)=[f,\psi_{-}](-1)=0\}.
\label{little delta_k}%
\end{equation}
The next theorem, established in \cite{Everitt-Littlejohn-Loveland} and
\cite{Loveland}, states properties of functions in $\delta_{K}.$ The proof of
this result is difficult and uses hard-analytic techniques together with an
integral inequality due to Chisholm and Everitt \cite{Chisholm and Everitt}.

\begin{theorem}
\label{Properties of little delta_k}Let $f,g\in\delta_{K}.$ Then

\begin{enumerate}
\item[(a)] $f^{\prime\prime\prime}\in L^{2}(-1,1)$ and $f,f^{\prime}%
,f^{\prime\prime}\in AC[-1,1];$

\item[(b)] $\lim_{x\rightarrow\pm1}((1-x^{2})^{3}f^{\prime\prime\prime
}(x))^{(j)}=0$ for $j=1,2;$

\item[(c)] $\lim_{x\rightarrow\pm1}(1-x^{2})f^{\prime\prime\prime}(x)=0;$

\item[(d)] $\lim_{x\rightarrow\pm1}(1-x^{2})f^{\prime\prime\prime}%
(x)\overline{g}^{\prime\prime}(x)=0;$

\item[(e)] $1\in\delta_{K}$ and
\begin{align}
\lim_{x\rightarrow1}[f,1]_{K}(x)  &  =-24f^{\prime\prime}(1)-24(A+1)f^{\prime
}(1)\label{[f,1](1)}\\
\lim_{x\rightarrow-1}[f,1]_{K}(x)  &  =24f^{\prime\prime}(-1)-24(B+1)f^{\prime
}(-1); \label{[f,1](-1)}%
\end{align}

\item[(f)] $(1-x^{2})\in\delta_{K}$ and
\begin{align*}
\lim_{x\rightarrow1}[f,1-x^{2}]_{K}(x)  &  =-48(A+2)f(1)\\
\lim_{x\rightarrow-1}[f,1-x^{2}]_{K}(x)  &  =48(B+2)f(-1);
\end{align*}

\item[(g)] $(1-x^{2})^{2}\in\delta_{K}$ and $\lim_{x\rightarrow\pm
1}[f,(1-x^{2})^{2}]_{K}(x)=\pm192f(\pm1);$

\item[(h)]
\begin{align*}
\lim_{x\rightarrow1}[f,g]_{K}(x)  &  =-24(f^{\prime\prime}(1)\overline
{g}(1)-\overline{g}^{\prime\prime}(1)f(1))\\
&  -24(A+1)(f^{\prime}(1)\overline{g}(1)-\overline{g}^{\prime}(1)f(1))\\
\lim_{x\rightarrow-1}[f,g]_{K}(x)  &  =24(f^{\prime\prime}(-1)\overline
{g}(-1)-\overline{g}^{\prime\prime}(-1)f(-1))\\
&  -24(B+1)(f^{\prime}(-1)\overline{g}(-1)-\overline{g}^{\prime}%
(-1)f(-1)).\blacksquare
\end{align*}

\end{enumerate}
\end{theorem}

From Theorem \ref{Properties of little delta_k}, the authors in
\cite{Everitt-Littlejohn-Loveland} and \cite{Loveland} define the operator
$T_{A,B}:\mathcal{D}(T_{A,B})\subset L_{\kappa}^{2}[-1,1]\rightarrow$
$L_{\kappa}^{2}[-1,1]$ by%
\begin{align}
T_{A,B}[f](x)  &  =\left\{
\begin{array}
[c]{ll}%
24Af^{\prime\prime}(-1)-24A(B+1)f^{\prime}(-1) & \text{if }x=-1\\
\ell_{K}[f](x) & \text{if }-1<x<1\\
24Bf^{\prime\prime}(1)+24B(A+1)f^{\prime}(1) & \text{if }x=1
\end{array}
\right. \label{Definition of T_A,B}\\
f  &  \in\mathcal{D}(T_{A,B}):=\delta_{K}.\nonumber
\end{align}

\begin{remark}
If $f\in C^{6}[-1,1]\cap\Delta_{K},$ calculations, using $($\ref{I-1}$),$
shows that
\begin{align*}
\ell_{K}[f](-1)  &  =24Af^{\prime\prime}(-1)-24A(B+1)f^{\prime}(-1)\\
\ell_{K}[f](1)  &  =24Bf^{\prime\prime}(1)+24B(A+1)f^{\prime}(1)
\end{align*}
so the definition of $T_{A,B}[f](\pm1)$ is consistent on smooth functions in
$\Delta_{K}.$
\end{remark}

\begin{theorem}
\label{T_A,B is self-adjoint} The operator $T_{A,B}$ is self-adjoint n
$L_{\kappa}^{2}[-1,1]$. Furthermore, the Krall polynomials $\{K_{n,A,B}%
\}_{n=0}^{\infty}$ form a complete set of eigenfunctions of $T_{A,B}.$

\begin{proof}
Details can be found in \cite[Chapter VII, Section 7.4]{Loveland}. The proof
consists of showing $T_{A,B}$ is symmetric whose range is all of $L_{\kappa
}^{2}[-1,1];$ by a well-known result $($see \cite[Chapter 1, Section
41]{Akhieser-Glazman}$),$ it follows that $T_{A,B}$ is self-adjoint. The
completeness of the polynomials follows from a general result in Szeg\"{o}
\cite[Theorem 3.1.5]{Szego}. Key to the surjectivity of $T_{A,B}$ is a careful
study of an auxiliary self-adjoint operator $S_{A,B}:\mathcal{D}%
(S_{A,B})\subset$ $L^{2}(-1,1)\rightarrow$ $L^{2}(-1,1)$ defined by%
\begin{align*}
S_{A,B}[f](x)  &  =\ell_{K}[f](x)\\
f\in\mathcal{D}(S_{A,B}):  &  =\{f\in\Delta_{K}\mid\lbrack f,\psi_{\pm}%
]_{K}(\pm1)=[f,1_{\pm}](\pm1)=0\}.
\end{align*}
Here $\psi_{\pm}$ are defined as in (\ref{psi+}) and (\ref{psi-}) while the
functions $1_{\pm}\in C^{6}[-1,1]\cap\Delta_{K}$ are given by%
\[
1_{+}(x)=\left\{
\begin{array}
[c]{cc}%
0 & \text{if }x\text{ is near }-1\\
1 & \text{if }x\text{ is near }1
\end{array}
\right.  \text{ and }1_{-}(x)=%
\begin{array}
[c]{cc}%
1 & \text{if }x\text{ is near }-1\\
0 & \text{if }x\text{ is near }1.
\end{array}
\]
The set of functions $\{\psi_{\pm},1_{\pm}\}$ are linearly independent modulo
the minimal domain $\mathcal{D}(T_{0}),$ and since the deficiency index of
$T_{0}$ is $4,$ it follows from the GKN Theorem \cite[Chapter V]{Naimark} that
$S_{A,B}$ is self-adjoint.
\end{proof}
\end{theorem}

Motivated by the Krall example and other higher-order differential equations
having orthogonal polynomial eigenfunctions, Littlejohn and Wellman developed
a generalized GKN theory \cite{Littlejohn-Wellman2019} that directly handles
these examples. Indeed, applying Theorem \ref{GKNEM}, given below, is
systematic and algorithmic. In the next section, we describe this theory and,
in Section \ref{Section Five}, we apply this new theory to further discuss the
Krall example.

\section{A Brief Review of the GKN-EM Theory\label{Section Four}}

In this section, we give a brief summary of the general GKN-EM theory
developed recently by Littlejohn and Wellman in \cite{Littlejohn-Wellman2019}
based on important earlier work of W. N. Everitt and L. Markus in
\cite{Everitt-Markus1999}, \cite{Everitt-MarkusBVP} and
\cite{Everitt-MarkusComplexSymplectic}. We remark that there are other
approaches, generalizing the GKN theory, in constructing self-adjoint
extensions of symmetric operators . One of these approaches is the general
theory of \textit{boundary triples}; we refer the reader to the recent
monograph \cite{Behrndt-Hassi-deSnoo} for details.

Throughout this section, we assume that $T_{0}$ and $T_{1}$ are densely
defined linear operators, with respective domains $\mathcal{D}(T_{0})$ and
$\mathcal{D}(T_{1}),$ in a Hilbert space $(H,\langle\cdot,\cdot\rangle_{H})$
satisfying the conditions:

\begin{enumerate}
\item[(a)] $T_{0}$ is a closed and symmetric operator;

\item[(b)] The deficiency indices of $T_{0}$ are equal and finite and denoted
by their common value \textrm{def}$(T_{0});$

\item[(c)] $T_{0}\subseteq T_{1}$ with $T_{0}^{\ast}=T_{1}$ and $T_{1}^{\ast
}=T_{0}.$
\end{enumerate}

We call the Hilbert space $(H,\langle\cdot,\cdot\rangle_{H})$ the \textit{base
space }and we will refer to the operators $T_{0}$ and $T_{1}$, respectively,
as the \textit{minimal operator} and \textit{maximal operators} since these
are the terms used in the GKN theory when both operators are generated by a
Lagrangian symmetric ordinary differential expression. To be clear, however,
it is not necessary in our situation that either of the operators $T_{0}$ and
$T_{1}$ are differential operators. We define the symplectic form
$[\cdot,\cdot]_{H}:\mathcal{D}(T_{1})$ $\mathrm{x}$ $\mathcal{D}%
(T_{1})\rightarrow\mathbb{C}$ by%
\begin{equation}
\lbrack x,y]_{H}:=\langle T_{1}x,y\rangle_{H}-\langle x,T_{1}y\rangle_{H}%
\quad(x,y\in\mathcal{D}(T_{1})). \label{Symplectic form for H}%
\end{equation}
Observe that (\ref{Symplectic form for H}) is a generalization of the
classical Green's formula (see, for example, (\ref{II-1.5})).

Let $(W,\langle\cdot,\cdot\rangle_{W})$ be a finite-dimensional complex
Hilbert space with
\begin{equation}
\dim W\leq\mathrm{def}(T_{0}). \label{IV-1}%
\end{equation}
We call $W$ the \textit{extension space. }Let $\{\xi_{j}\mid j=1,\ldots,\dim
W\}$ be an orthonormal basis of $W$ and let $\mathcal{B}:W\rightarrow W$ be a
fixed, self-adjoint operator.

The direct sum space $H\oplus W=\{(x,a)\mid x\in H,a\in W\}$ is called the
\textit{extended space}; it is well known that $H\oplus W$ is a Hilbert space
when endowed with the inner product
\[
\langle(x,a),(y,b)\rangle_{H\oplus W}:=\langle x,y\rangle_{H}+\langle
a,b\rangle_{W}.
\]
It is this Hilbert space $H\oplus W$ that we seek to extend the classical GKN
theory (see \cite{Naimark}).

As with the classic GKN theory presented in \cite[Chapter V]{Naimark}, we call
a collection of vectors $\{t_{j}\mid j=1,\ldots,\mathrm{def}(T_{0}%
)\}\subseteq\mathcal{D}(T_{1})$ is called a \textit{GKN set} for $T_{0}$ if

\begin{enumerate}
\item[(i)] the set $\{t_{j}\mid j=1,\ldots,\mathrm{def}(T_{0})\}$ is linearly
independent modulo the minimal domain $D(T_{0})$; that is to say:\textit{ }%
\[
\text{if }\sum_{j=1}^{\mathrm{def}(T_{0})}\alpha_{j}t_{j}\in\mathcal{D}%
(T_{0})\text{ then }\alpha_{j}=0\text{ for }j=1,\ldots,\mathrm{def}(T_{0});
\]
and

\item[(ii)] the set $\{t_{j}\mid j=1,\ldots,\mathrm{def}(T_{0})\}$ satisfies
the symmetry conditions
\[
\lbrack t_{i},t_{j}]_{H}=0\quad(i,j=1,\ldots,\mathrm{def}(T_{0})).
\]

\end{enumerate}

If $G\subseteq D(T_{1})$ is a GKN set for $T_{0}$, then a non-empty, proper
subset $P\subseteq G$ is called a \textit{partial GKN set} for $T_{0}$. As we
will see, partial GKN sets play an important role in our generalized GKN
theory in $H\oplus W.$

We now list several definitions necessary before we state the main result
(Theorem \ref{GKNEM}) below.

\begin{enumerate}
\item $P=\{t_{j}\mid j=1,\ldots,\dim W\}$ is a partial GKN set for $T_{0}$;

\item $\Phi_{0}:=\mathcal{D}(T_{0})+\mathrm{span}\{t_{j}\mid j=1,\ldots,\dim
W\}$;

\item $\Psi:\Phi_{0}\rightarrow W$ is defined to be
\[
\Psi\left(  x_{0}+\sum_{j=1}^{\dim W}\alpha_{j}t_{j}\right)  :=\sum
_{j=1}^{\dim W}\alpha_{j}\xi_{j}\quad(x_{0}\in\mathcal{D}(T_{0}));
\]

\item $\Omega:\mathcal{D}(T_{1})\rightarrow W$ is given by
\[
\Omega x:=\sum_{j=1}^{\dim W}[x,t_{j}]_{H}\xi_{j}\quad(x\in\mathcal{D}%
(T_{1}));
\]

\item $\widehat{T}_{1}:\mathcal{D}(\widehat{T}_{1})\subseteq H\oplus
W\rightarrow H\oplus W$ is the maximal operator in $H\oplus W$ defined by
\begin{align*}
\widehat{T}_{1}(x,a)  &  :=(T_{1}x,\mathcal{B}a-\Omega x)\\
\mathcal{D}(\widehat{T}_{1})  &  :=\{(x,a)\mid x\in\mathcal{D}(T_{1});\;a\in
W\};
\end{align*}
It is clear from the definition of $\mathcal{D}(\widehat{T}_{1})$ why
$\widehat{T}_{1}$ is called `maximal'. Note that $\widehat{T}_{1}$ is
dependent upon a partial GKN set and the choice of self-adjoint operator
$\mathcal{B}$ in $W.$ Consequently, there is a continuum of such maximal
operators in $H\oplus W.$

\item $\widehat{T}_{0}:\mathcal{D}(\widehat{T}_{0})\subseteq H\oplus
W\rightarrow H\oplus W$ is the so-called \textit{minimal operator} in $H\oplus
W$ defined by
\begin{align*}
\widehat{T}_{0}(x,\Psi x)  &  :=(T_{1}x,\mathcal{B}\Psi x)\\
\mathcal{D}(\widehat{T}_{0})  &  :=\{(x,\Psi x)\mid x\in\Phi_{0}\};
\end{align*}
(see Theorem \ref{GKNEM}) below for the justification of the term `minimal'
for $\widehat{T}_{0}$)

\item $[\cdot,\cdot]_{H\oplus W}$ is the symplectic form defined by
\begin{align}
\lbrack(x,a),(y,b)]_{H\oplus W}:  &  =\langle\widehat{T}_{1}(x,a),(y,b)\rangle
_{H\oplus W}-\langle(x,a),\widehat{T}_{1}(y,b)\rangle_{H\oplus W}%
\label{IV-2.5}\\
&  =[x,y]_{H}-\langle\Omega x,b\rangle_{W}+\langle a,\Omega y\rangle_{W}%
,\quad(((x,a),(y,b))\in\mathcal{D}(\widehat{T}_{1})); \label{IV-3}%
\end{align}
the equality of (\ref{IV-2.5}) and (\ref{IV-3}) follows from the
self-adjointness of $\mathcal{B}$.
\end{enumerate}

\begin{theorem}
\label{Minimal Operator T_0 hat}Under the above definitions and assumptions,
\end{theorem}

\begin{enumerate}
\item[(a)] $\widehat{T}_{0}$ is a closed, symmetric operator satisfying
$\widehat{T}_{0}\subseteq\widehat{T}_{1}$ with $(\widehat{T}_{0})^{\ast
}=\widehat{T}_{1}$ and $(\widehat{T}_{1})^{\ast}=\widehat{T}_{0}$;

\item[(b)] The deficiency indices of $\widehat{T}_{0}$ are equal, finite and
satisfy $\mathrm{def}(\widehat{T}_{0})=\mathrm{def}(T_{0})$.
\end{enumerate}

From part (b) of this theorem, we see that the minimal operator $\widehat{T}%
_{0}$ (respectively, the maximal operator $\widehat{T}_{1})$ has self-adjoint
extensions (respectively, self-adjoint restrictions) in $H\oplus W.$ The
following theorem is a generalization of the classic Glazman-Krein-Naimark
theorem for the setting $H\oplus W$ with one important difference: the
operators $\widehat{T}_{0}$ and $\widehat{T}_{1}$ are not assumed to be
differential operators. The authors in \cite{Littlejohn-Wellman2019} refer to
this generalization as the GKN-EM Theorem.

\begin{theorem}
\label{GKNEM} \quad\newline

\begin{enumerate}
\item[(a)] Suppose $\widehat{T}$ is a self-adjoint extension of $\widehat{T}%
_{0}$ satisfying $\widehat{T}_{0}\subseteq\widehat{T}\subseteq\widehat{T}_{1}%
$. Then there exists a GKN set $\{(x_{j},a_{j})\mid j=1,\ldots,\mathrm{def}%
(T_{0})\}\subseteq\mathcal{D}(\widehat{T}_{1})$ such that
\begin{align}
\widehat{T}(x,a)  &  =(T_{1}x,\mathcal{B}a-\Omega x)\label{IV-4}\\
\mathcal{D}(\widehat{T})  &  =\{(x,a)\in\mathcal{D}(\widehat{T}_{1}%
)\mid\lbrack(x,a),(x_{j},a_{j})]_{H\oplus W}=0\;(j,1,\ldots,\mathrm{def}%
(T_{0}))\}. \label{IV-5}%
\end{align}

\item[(b)] If $\widehat{T}$ is defined as in $($\ref{IV-3}$)$ and
$($\ref{IV-4}$),$ where $\{(x_{j},a_{j})\mid j=1,\ldots,\mathrm{def}%
(T_{0})\}\subseteq\mathcal{D}(\widehat{T}_{1})$ is a GKN, then $\widehat{T}$
is a self-adjoint extension of $\widehat{T}_{0}$ in $H\oplus W$.
\end{enumerate}
\end{theorem}

In the next section, we use this theorem to construct the self-adjoint
operator $\widehat{T}_{A,B}$, generated by the Krall differential expression
$\ell_{K}[\cdot],$ in
\[
L_{\kappa}^{2}[-1,1]\text{ }\simeq\text{ }L^{2}(-1,1)\oplus\mathbb{C}^{2}%
\]
having the Krall polynomials as eigenfunctions.

\section{Application of the GKN-EM Theorem to the Krall Differential
Expression\label{Section Five}}

Full details and proofs of results given below can be found in the recent
Ph.D. thesis of Elliott \cite{Elliott}.

For the Krall example, the base space is $H=L^{2}(-1,1)$ equipped with inner
product%
\[
\left\langle f,g\right\rangle _{H}:=\int_{-1}^{1}f(x)\overline{g}%
(x)dx\quad(f,g\in L^{2}(-1,1)).
\]
The maximal operator $T_{1}$ and the minimal operator $T_{0}$ in $H$
associated with $\ell_{K}[\cdot]$ are defined, respectively, by
\begin{align*}
(T_{1}f)(x)  &  =\ell_{K}[f](x)\quad(\text{a.e}.x\in(-1,1))\\
f\in\mathcal{D}(T_{1})  &  =\Delta_{K},
\end{align*}
where $\Delta_{K}$ is given in (\ref{II-1}), and
\begin{align*}
(T_{0}f)(x)  &  =\ell_{K}[f](x)\quad(\text{a.e. }x\in(-1,1))\\
f\in\mathcal{D}(T_{0})  &  =\{f\in\Delta_{K}|\left[  f,g\right]  _{H}%
=0\;(g\in\Delta_{K})\}.
\end{align*}
We remind the reader that, for $f,g\in\Delta_{K}$, we have
\begin{align}
\left[  f,g\right]  _{H}  &  :=\left\langle T_{1}f,g\right\rangle
_{H}-\left\langle f,T_{1}g\right\rangle _{H}\label{sesquilinear form in H}\\
&  =[f,g]_{K}(1)-[f,g]_{K}(-1),\nonumber
\end{align}
where $[\cdot,\cdot]_{K}$ is the symplectic form defined in (\ref{II-2}). As
is well known, $T_{0}$ is closed and symmetric in $L^{2}(-1,1)$ with
$T_{0}^{\ast}=T_{1}$ and $T_{1}^{\ast}=T_{0}$. The deficiency indices of
$T_{0}$ are equal and finite and, as the Frobenius analysis from Section
\ref{Section Two} shows,%
\begin{equation}
\mathrm{def}(T_{0})=4. \label{def(T_0)}%
\end{equation}

\noindent For the Krall example, the extension space is $W=\mathbb{C}^{2}$
with inner product
\begin{equation}
\left\langle (a,b),(a^{\prime},b^{\prime})\right\rangle _{W}:=\frac
{a\overline{a}^{\prime}}{A}+\frac{b\overline{b}^{\prime}}{B}. \label{IP in W}%
\end{equation}
Define $\xi_{1}=(\sqrt{A},0)$ and $\xi_{2}=(0,\sqrt{B})$. Then $\{\xi_{1}%
,\xi_{2}\}$ is an orthonormal basis for $W$. For the self-adjoint operator
$\mathcal{B}:W\rightarrow W,$ we choose $\mathcal{B}=0$ (see Remark
(\ref{B=0}) below for an explanation).

The inner product (\ref{IV-3}) in the extended space $H\oplus W$ is
specifically given by
\begin{align}
\left\langle (f,(a_{1},b_{1})),(g,(a_{2},b_{2}))\right\rangle _{H\oplus W}  &
=\langle f,g\rangle_{H}+\langle(a_{1},b_{1}),(a_{2},b_{2})\rangle
_{W}\label{sesquiliner form in H+W}\\
&  =\frac{a_{1}\overline{a}_{2}}{A}+\int_{-1}^{1}f(x)\overline{g}%
(x)dx+\frac{b_{1}\overline{b}_{2}}{B};\nonumber
\end{align}
compare (\ref{sesquiliner form in H+W}) to the inner product in (\ref{I-7}).
For ease of eye, we adopt the notation $(f,a,b)=(f,(a,b))$ to represent a
vector from $H\oplus W.$

Define $P=\{t_{1},t_{2}\}\subset C^{6}[-1,1]\cap\Delta_{K}$ by
\begin{equation}
t_{1}(x)=%
\begin{cases}
\sqrt{A} & x\text{ near }-1\\
0 & x\text{ near }1
\end{cases}
,\text{ }t_{2}(x)=%
\begin{cases}
0 & x\text{ near }-1\\
\sqrt{B} & x\text{ near }1.
\end{cases}
\label{partial gkn set}%
\end{equation}
We leave it to the reader to prove the following lemma.

\begin{lemma}
$P$ is a partial GKN set for $T_{0}$.
\end{lemma}

In line with the notation and theory from Section \ref{Section Four}, we
define
\[
\Phi_{0}:=\mathcal{D}(T_{0})+\mathrm{span}\{t_{1},t_{2}\},
\]
and the operators, $\Psi:\Phi_{0}\rightarrow W$ by
\[
\Psi(f_{0}+\alpha_{1}t_{1}+\alpha_{2}t_{2}):=\alpha_{1}\xi_{1}+\alpha_{2}%
\xi_{2}=\left(  \alpha_{1}\sqrt{A},\alpha_{2}\sqrt{B}\right)  \quad(f_{0}%
\in\mathcal{D}(T_{0})),
\]
and $\Omega:\Delta_{K}\rightarrow W$ by
\begin{align}
\Omega f  &  :=[f,t_{1}]_{H}\xi_{1}+[f,t_{2}]_{H}\xi_{2}\label{Omega[f]}\\
&  =-\sqrt{A}[f,1]_{K}(-1)\xi_{1}+\sqrt{B}[f,1]_{K}(1)\xi_{2}=\left(
-A[f,1]_{K}(-1),B[f,1]_{K}(1)\right)  ,\nonumber
\end{align}
where $[f,g]_{K}(\pm1)$ are defined in (\ref{Limit of [f,g] at 1 and -1}).

With $\mathcal{B}=0$, the minimal operator $\widehat{T}_{0}:\mathcal{D(}%
\widehat{T}_{0})\subseteq H\oplus W\rightarrow H\oplus W$ is given by
\begin{align*}
\widehat{T}_{0}\left(  f,(a,b)\right)   &  =\left(  T_{1}f,0,0\right) \\
\mathcal{D(}\widehat{T}_{0})  &  =\left\{  (f,\Psi f)\mid f\in\Phi
_{0}\right\}  .
\end{align*}
We note, from (\ref{def(T_0)}) and Theorem \ref{Minimal Operator T_0 hat},
that $\widehat{T}_{0}$ is a closed symmetric operator in $H\oplus W$ with
$\mathrm{def}(\widehat{T}_{0})=4.$

The associated maximal operator $\widehat{T}_{1}$ in $H\oplus W$ is defined
by
\begin{align*}
\widehat{T}_{1}\left(  f,a,b\right)   &  =\left(  T_{1}f,-\Omega f\right)
=\left(  T_{1}f,A[f,1]_{K}(-1),-B[f,1]_{K}(1)\right) \\
\mathcal{D(}\widehat{T}_{1})  &  =\{\left(  f,a,b\right)  \mid f\in\Delta
_{K},(a,b)\in W\}.
\end{align*}

The symplectic form $[\cdot,\cdot]_{H\oplus W}$, defined in (\ref{IV-3}), is
given by
\[
\lbrack(f,a_{1},b_{1}),(g,a_{2},b_{2})]_{H\oplus W}=[f,g]_{H}-\left\langle
\Omega f,(a_{2},b_{2})\right\rangle _{W}+\left\langle (a_{1},b_{1}),\Omega
g\right\rangle _{W},
\]
for $(f,a_{1},b_{1}),(g,a_{2},b_{2})\in\mathcal{D(}\widehat{T}_{1})$.

We now construct a GKN set for $\widehat{T}_{0}$ in $H\oplus W.$ Define the
functions $y_{1},y_{2},y_{3},y_{4}\in C^{6}[-1,1]\cap\Delta_{K}$ by
\begin{align}
y_{1}(x)  &  =%
\begin{cases}
0 & x\text{ near }-1\\
(1-x^{2})^{2} & x\text{ near }1,
\end{cases}
\quad y_{2}(x)=%
\begin{cases}
(1-x^{2})^{2} & x\text{ near }-1\\
0 & x\text{ near }1,
\end{cases}
\label{y_j's}\\
y_{3}(x)  &  =%
\begin{cases}
0 & x\text{ near }-1\\
1-x^{2} & x\text{ near }1,
\end{cases}
\text{ }y_{4}(x)=%
\begin{cases}
1-x^{2} & x\text{ near }-1\\
0 & x\text{ near }1.
\end{cases}
\quad\nonumber
\end{align}
Calculations, using Theorem \ref{Properties of Maximal Domain}, show that, for
$f\in\Delta_{K}$,
\begin{align}
\lbrack f,y_{1}]_{H}  &  =[f,(1-x^{2})^{2}]_{K}(1)=192f(1),\nonumber\\
\lbrack f,y_{2}]_{H}  &  =-[f,(1-x^{2})^{2}]_{K}(-1)=192f(-1),\nonumber\\
\lbrack f,y_{3}]_{H}  &  =[f,1-x^{2}]_{K}(1)=2\Lambda\lbrack
f](1)-48(A+2)f(1),\label{[f,y_j]_H}\\
\lbrack f,y_{4}]_{H}  &  =-[f,1-x^{2}]_{K}(-1)=2\Lambda\lbrack
f](-1)-48(B+2)f(-1).\nonumber
\end{align}
The importance of these functions $\{y_{i}\}_{i=0}^{4}$ lies in the following result.

\begin{lemma}
\label{GKN Set for T_0 hat}$\{y_{i},(0,0))\}_{i=1}^{4}$ is a GKN set for
$\widehat{T}_{0}$ in $H\oplus W$.
\end{lemma}

\begin{proof}
It will first be shown that $\{y_{i},(0,0))\}_{i=1}^{4}$ is linearly
independent modulo $\mathcal{D(}\widehat{T}_{0})$. Since
\[
\lbrack(y_{i},(0,0)),(y_{j},(0,0))]_{H\oplus W}=[y_{i},y_{j}]_{H}%
\]
for $i,j=1,2,3,4$, it is sufficient to show that $\{y_{i}\}_{i=1}^{4}$ is
linearly independent modulo $\mathcal{D}(T_{0})$. To this end, let $f\in
\Delta_{K}$ and arbitrarily choose $c_{1},c_{2},c_{3},c_{4}\in\mathbb{C}$.
Then, from Theorem \ref{Properties of Maximal Domain} (iv) and (v), we find that%

\begin{align}
\lbrack f,c_{1}y_{1}+c_{2}y_{2}+c_{3}y_{3}+c_{4}y_{4}]_{H}=  &  \overline
{c}_{1}[f,y_{1}]_{H}+\overline{c}_{2}[f,y_{2}]_{H}+\overline{c}_{3}%
[f,y_{3}]_{H}+\overline{c}_{4}[f,y_{4}]_{H}\nonumber\\
=  &  192\overline{c}_{1}f(1)+192\overline{c}_{2}f(-1)\nonumber\\
&  +\overline{c}_{3}\left(  2\Lambda\lbrack f](1)-48(A+2)f(1)\right)
\label{GKN calculation}\\
&  +\overline{c}_{4}\left(  2\Lambda\lbrack f](-1)-48(B+2)f(-1)\right)
.\nonumber
\end{align}

\noindent We now specifically choose various $f\in\Delta_{K}.$ Define
$f_{1},f_{2}\in C^{6}[-1,1]\cap\Delta_{K}$ by
\[
f_{1}(x):=%
\begin{cases}
1 & x\text{ near }-1\\
0 & x\text{ near }1,
\end{cases}
,\quad f_{2}(x):=%
\begin{cases}
0 & x\text{ near }-1\\
1 & x\text{ near }1.
\end{cases}
\]

\noindent Then
\begin{equation}
\lbrack f_{1},c_{1}y_{1}+c_{2}y_{2}+c_{3}y_{3}+c_{4}y_{4}]_{H}=192\overline
{c}_{2}-48(B+2)\overline{c}_{4} \label{f_1 calculation}%
\end{equation}
and%
\begin{equation}
\lbrack f_{2},c_{1}y_{1}+c_{2}y_{2}+c_{3}y_{3}+c_{4}y_{4}]_{H}=192\overline
{c}_{1}-48(A+2)\overline{c}_{3}.\newline\label{f_2 calculation}%
\end{equation}

Using the functions $h_{\pm}(x)$ defined in Theorem
\ref{Properties of Maximal Domain} (vi), some tedious calculations show that
$\Lambda\lbrack h_{+}](1)=24$ and
\begin{equation}
\lbrack h_{+},c_{1}y_{1}+c_{2}y_{2}+c_{3}y_{3}+c_{4}y_{4}]_{H}=48\overline
{c}_{3}; \label{h_+ calculation}%
\end{equation}

\noindent\noindent similarly, $\Lambda\lbrack h_{-}](-1)=24$ and
\begin{equation}
\lbrack h_{-},c_{1}y_{1}+c_{2}y_{2}+c_{3}y_{3}+c_{4}y_{4}]_{H}=48\overline
{c}_{4}. \label{h_- calculation}%
\end{equation}
From the identities in (\ref{f_1 calculation}), (\ref{f_2 calculation}),
(\ref{h_+ calculation}), and (\ref{h_- calculation}), we immediately see that
the simultaneous solution of the system of equations
\begin{align*}
\lbrack f_{1},c_{1}y_{1}+c_{2}y_{2}+c_{3}y_{3}+c_{4}y_{4}]_{H}  &  =0,\\
\lbrack f_{2},c_{1}y_{1}+c_{2}y_{2}+c_{3}y_{3}+c_{4}y_{4}]_{H}  &  =0,\\
\lbrack h_{+},c_{1}y_{1}+c_{2}y_{2}+c_{3}y_{3}+c_{4}y_{4}]_{H}  &  =0,\\
\lbrack h_{-},c_{1}y_{1}+c_{2}y_{2}+c_{3}y_{3}+c_{4}y_{4}]_{H}  &  =0
\end{align*}
is $\overline{c}_{1}=\overline{c}_{2}=\overline{c}_{3}=\overline{c}_{4}=0$.
Therefore, $\{y_{i}\}_{i=1}^{4}$ is linearly independent modulo $\mathcal{D}%
(T_{0})$ and, hence, $\{y_{i},(0,0))\}_{i=1}^{4}$ is linearly independent
modulo $\mathcal{D(}\widehat{T}_{0})$.

We now show that $[(y_{i},(0,0)),(y_{j},(0,0))]_{H\oplus W}=[y_{i},y_{j}%
]_{H}=0$ for $i,j=1,2,3,4$. Straightforward calculations show that
\begin{align*}
\lbrack y_{1},y_{1}]_{H}  &  =192y_{1}(1)=0,\quad\lbrack y_{1},y_{2}%
]_{H}=192y_{1}(-1)=0,\\
\lbrack y_{1},y_{3}]_{H}  &  =-192y_{3}(1)=0,\quad\lbrack y_{1},y_{4}%
]_{H}=-192y_{4}(1)=0,\\
\lbrack y_{2},y_{2}]_{H}  &  =192y_{2}(-1)=0,\quad\lbrack y_{2},y_{3}%
]_{H}=-192y_{3}(-1)=0,\\
\lbrack y_{2},y_{4}]_{H}  &  =-192y_{4}(-1)=0,\quad\lbrack y_{3},y_{3}%
]_{H}=2\Lambda\lbrack y_{3}](1)-48(A+2)y_{3}(1)=0,\\
\lbrack y_{3},y_{4}]_{H}  &  =2\Lambda\lbrack y_{3}](-1)-48(B+2)y_{3}%
(-1)=0,\quad\lbrack y_{4},y_{4}]_{H}=2\Lambda\lbrack y_{4}](-1)-48(B+2)y_{4}%
(-1)=0.
\end{align*}
and since $[y_{i},y_{j}]_{H}=-[y_{j},y_{i}]_{H}$, we conclude that
$\{(y_{i},(0,0))\}_{i=1}^{4}$ is a GKN set for $\widehat{T}_{0}$.
\end{proof}

It follows, from Theorem \ref{GKNEM}, that the operator $\widehat{T}%
_{A,B}:\mathcal{D}(\widehat{T}_{A,B})\subset L^{2}(-1,1)\oplus\mathbb{C}%
^{2}\rightarrow$ $L^{2}(-1,1)\oplus\mathbb{C}^{2}$ defined by%
\begin{align}
\widehat{T}_{A,B}[(f,a,b)](x) &  =(\ell_{K}[f](x),-\Omega\lbrack
f])\quad(\text{a.e. }x\in(-1,1))\label{Form of T_A,B hat}\\
\mathcal{D}(\widehat{T}_{A,B}) &  :=\{(f,a,b)\in\mathcal{D}(\widehat{T}%
_{1})\mid\lbrack(f,(a,b)),(y_{j},(0,0))]_{H\oplus W}%
=0\;(j,1,2,3,4)\}\label{D(T_A,B hat)}%
\end{align}
is self-adjoint in $L^{2}(-1,1)\oplus\mathbb{C}^{2}.$ For the remainder of
this section, we will prove this operator is equivalent to the self-adjoint
operator found by Loveland et al in \cite{Everitt-Littlejohn-Loveland} and
\cite{Loveland} and to show that the Krall polynomials $\{K_{n,A,B}%
\}_{n=0}^{\infty}$ are a complete set of eigenfunctions of $\widehat{T}%
_{A,B}.$ To prove this, we begin with the following fundamental theorem.

\begin{theorem}
\label{GKNEM domain is little delta}With $\{y_{i}\}_{i=1}^{4}$ defined in
$($\ref{y_j's}$)$ and $\mathcal{D}(\widehat{T}_{A,B})$ as given in
$($\ref{D(T_A,B hat)}$),$ we have
\begin{equation}
\mathcal{D}(\widehat{T}_{A,B})=\{(f,f(-1),f(1))\mid f\in\delta_{K}%
\},\label{Domain of T_A,B}%
\end{equation}
where $\delta_{K}$ is defined in $($\ref{little delta_k}$).$ Moreover, for
$f\in\mathcal{D}(\widehat{T}_{A,B}),$
\begin{align}
&  \widehat{T}_{A,B}[(f,f(-1),f(1))](x)\label{Explicit Form of T_A,B}\\
&  =(\ell_{K}[f](x),24Af^{\prime\prime}(-1)-24A(B+1)f^{\prime}%
(-1),24Bf^{\prime\prime}(1)+24B(A+1)f^{\prime}(1)).\nonumber
\end{align}

\begin{proof}
Suppose $f\in\Delta_{K}$ and $(a,b)\in\mathbb{C}^{2}$ satisfy, for
$j=1,2,3,4,$%
\begin{align}
0  &  =[(f,(a,b)),(y_{j},(0,0))]_{H\oplus W}\label{SA BC's}\\
&  =[f,y_{j}]_{H}+\langle(a,b),\Omega y_{j}\rangle_{W}.\nonumber
\end{align}
Calculations show that%
\begin{align}
\Omega y_{1}  &  =(0,-192B),\quad\Omega y_{2}=(-192A,0),\nonumber\\
\Omega y_{3}  &  =(0,48B(A+2)),\quad\Omega y_{4}=(48A(B+2),0).
\label{Omega[y_j]}%
\end{align}
Together with the identities in $($\ref{[f,y_j]_H}$),$ it follows that the
equations in $($\ref{SA BC's}$)$ are equivalent to%
\begin{equation}%
\begin{tabular}
[c]{l}%
$(\text{i})\text{ }192f(1)-192b=0,$\\
$(\text{ii})\text{ }192f(-1)-192a=0,$\\
$(\text{iii})\text{ }2\Lambda\lbrack f](1)-48(A+2)f(1)+48b(A+2)=0,$\\
$(\text{iv})\text{ }2\Lambda\lbrack f](-1)-48(B+2)f(-1)+48a(B+2)=0.$%
\end{tabular}
\ \ \ \ \ \label{Equivalent conditions}%
\end{equation}
From $($i$)$ and $($ii$)$ above, we find that $f(1)=b$ and $f(-1)=a.$
Substituting these values into $($iii$)$ and $($iv$),$ we find that%
\begin{equation}
\Lambda\lbrack f](-1)=\Lambda\lbrack f](1)=0.
\label{Capital Lambda conditions}%
\end{equation}
From $($\ref{GKN BC's}$)$ and $($\ref{little delta_k}$),$ we see that
$f\in\delta_{K}$. It follows that
\begin{align}
\{(f,(a,b))  &  \in\mathcal{D}(\widehat{T}_{1})\mid\lbrack(f,(a,b)),(y_{j}%
,(0,0))]_{H\oplus W}=0\;(j,1,2,3,4)\}\label{Inclusion 1}\\
&  \subseteq\{(f,f(-1),f(1))\mid f\in\delta_{K}\}.\nonumber
\end{align}
The reverse inclusion follows from the conditions given in $($%
\ref{Equivalent conditions}$).$ As for the form of $\widehat{T}_{A,B}$ note
that, from Theorem \ref{GKNEM}, the form of $\widehat{T}_{A,B}$ is given in
$($\ref{Form of T_A,B hat}$).$ Using $($\ref{Omega[f]}$)$ together with
$($\ref{[f,1](1)}$)$ and $($\ref{[f,1](-1)}$),$ we obtain the representation
given in $($\ref{Explicit Form of T_A,B}$).$
\end{proof}
\end{theorem}

\begin{remark}
Compare Theorem \ref{GKNEM domain is little delta} to Theorem
\ref{T_A,B is self-adjoint} to see that $T_{A,B}$, given in $($%
\ref{Definition of T_A,B}$),$ and $\widehat{T}_{A,B},$ given in Theorem
\ref{GKNEM domain is little delta}, are $($essentially$)$ the same.
\end{remark}

\begin{remark}
Note that, from $($i$)$ and $($ii$)$ in $($\ref{Equivalent conditions}$),$ two
of the GKN-EM boundary conditions, specifically
\[
\lbrack f,y_{i}]_{H}+\langle(a,b),\Omega y_{i}\rangle_{W}=0\quad(i=1,2),
\]
are equivalent to the conditions $a=f(-1)$ and $b=f(1).$ In other words, these
boundary conditions force continuity at the endpoints of $f\in\delta_{K}$. The
other two boundary conditions are, essentially, classical GKN boundary
conditions; indeed, once continuity at the endpoints is forced, these
remaining two conditions, namely
\[
\lbrack f,y_{j}]_{H}+\langle(a,b),\Omega y_{j}\rangle_{W}=0\quad(i=3,4)
\]
are equivalent to%
\[
\lbrack f,\psi_{+}]_{H}=[f,\psi_{-}]_{H}=0.
\]

\end{remark}

The final result that we establish shows that the Krall polynomials
$\{K_{n,A,B}\}_{n=0}^{\infty}$ are eigenfunctions of $\widehat{T}_{A,B}%
[\cdot]$ in the following sense.

\begin{theorem}
$\{(K_{n,A,B},$ $K_{n,A,B}(-1),$ $K_{n,A,B}(1)\}_{n=0}^{\infty}$ form a
complete set of eigenfunctions of $\widehat{T}_{A,B}$ in $L^{2}((-1,1)\oplus
\mathbb{C}^{2}.$

\begin{proof}
Clearly $\{K_{n,A,B}\}_{n=0}^{\infty}$ $\subset\Delta_{K}.$ Straightforward
calculations show that%
\[
\lbrack K_{n,A,B},\psi_{-1}]_{H}=-[K_{n,A,B},\psi_{-1}]_{K}(-1)=0
\]
and%
\[
\lbrack K_{n,A,B},\psi_{+1}]_{H}=[K_{n,A,B},\psi_{+1}]_{K}(1)=0
\]
so $\{K_{n,A,B}\}_{n=0}^{\infty}$ $\subset\delta_{K}.$ The Krall polynomials
are eigenfunctions of the Krall expression $\ell_{K}[\cdot]$; specifically,
from $($\ref{I-4}$),$ we see that
\[
\ell_{K}[K_{n,A,B}](x)=\lambda_{n}K_{n,A,B}(x)\quad(x\in\lbrack-1,1]),
\]
where the eigenvalue $\lambda_{n}$ is given in $($\ref{I-3}$).$ Moreover, we
can see directly from the Krall expression in $($\ref{I-1}$)$ that%
\[
\ell_{K}[K_{n,A,B}](-1)=24AK_{n,A,B}^{\prime\prime}(-1)-(24AB+24A)K_{n,A,B}%
^{\prime}(-1)=\lambda_{n}K_{n,A,B}(-1)
\]
and
\[
\ell_{K}[K_{n,A,B}](1)=24BK_{n,A,B}^{\prime\prime}(1)+(24AB+24B)K_{n,A,B}%
^{\prime}(1)=\lambda_{n}K_{n,A,B}(1).
\]
Then, from $($\ref{Explicit Form of T_A,B}$),$
\begin{align*}
&  \widehat{T}_{A,B}[(K_{n,A,B},K_{n,A,B}(-1),K_{n,A,B}(1))]\\
&  =(\ell_{K}[K_{n,A,B}],24AK_{n,A,B}^{\prime\prime}(-1)-24A(B+1)K_{n,A,B}%
^{\prime}(-1),24BK_{n,A,B}^{\prime\prime}(1)+24B(A+1)K_{n,A,B}^{\prime}(1))\\
&  =\left(  \ell_{K}[K_{n,A,B}],\ell_{K}[K_{n,A,B}](-1),\ell_{K}%
[K_{n,A,B}](1)\right) \\
&  =\lambda_{n}\left(  K_{n,A,B},K_{n,A,B}(-1),K_{n,A,B}(1)\right)  .
\end{align*}
The completeness of $\{(K_{n,A,B},$ $K_{n,A,B}(-1),$ $K_{n,A,B}(1)\}_{n=0}%
^{\infty}$ follows from the self-adjointness of $\widehat{T}_{A,B}$ and the
discreteness of its spectrum $\sigma(\widehat{T}_{A,B})=\{\lambda_{n}\mid
n\in\mathbb{N}_{0}\}.$
\end{proof}
\end{theorem}

\begin{remark}
\label{B=0}We began our analysis of the Krall expression in this section by
assuming the self-adjoint operator $\mathcal{B}$ in $W=\mathbb{C}^{2}$ is the
zero matrix. This is necessary for our $\widehat{T}_{A,B}$ to match up with
Loveland's self-adjoint operator $($\ref{Definition of T_A,B}$)$. In
particular, the condition that $\mathcal{B}=0$ is necessary for the Krall
polynomials to be eigenfunctions of the operator $\widehat{T}_{A,B}.$
\end{remark}

\end{document}